\documentclass[11pt,reqno]{amsart}
\usepackage{cmap,mathtools}
\usepackage[T1]{fontenc}
\usepackage[utf8]{inputenc}
\usepackage[english]{babel}
\usepackage{amsfonts,amssymb,amsmath}
\usepackage{footnote,bigints}
 \usepackage{enumitem}

\allowdisplaybreaks
\usepackage[pagewise, mathlines]{lineno}
\usepackage{graphicx}
\usepackage{epstopdf}
\usepackage{a4wide}
\usepackage{xcolor}
\usepackage[colorlinks=true,linktocpage,pdfpagelabels,
bookmarksnumbered,bookmarksopen]{hyperref}
\definecolor{ForestGreen}{rgb}{0.1,0.6,0.05}
\definecolor{EgyptBlue}{rgb}{0.063,0.1,0.6}
\hypersetup{
	colorlinks=true,
	linkcolor=EgyptBlue,         
	citecolor=ForestGreen,
	urlcolor=olive
}

\usepackage[hyperpageref]{backref}

\newtheorem{theorem}{Theorem}[section]
\newtheorem{proposition}{Proposition}[section]
\newtheorem{definition}{Definition}[section]
\newtheorem{lemma}{Lemma}[section]
\newtheorem{remark}{Remark}[section]
\newtheorem{cor}{Corollary}[section]

\numberwithin{equation}{section}
\numberwithin{theorem}{section}
\numberwithin{equation}{section}
\numberwithin{theorem}{section}

\DeclarePairedDelimiter\norm{\lVert}{\rVert}

\usepackage{ulem}
\usepackage{xcolor}
\usepackage[colorlinks=true,linktocpage,pdfpagelabels,
bookmarksnumbered,bookmarksopen]{hyperref}
\definecolor{ForestGreen}{rgb}{0.1,0.6,0.05}
\definecolor{EgyptBlue}{rgb}{0.063,0.1,0.6}
\hypersetup{
	colorlinks=true,
	linkcolor=EgyptBlue,         
	citecolor=ForestGreen,
	urlcolor=olive
}

\usepackage[hyperpageref]{backref}
\usepackage[foot]{amsaddr}

\title [Isoperimetric bounds for higher Steklov-Dirichlet eigenvalues]{Sharp bounds for higher Steklov-Dirichlet eigenvalues on domains with spherical holes}
\author{Sagar Basak$^1$ \and Anisa Chorwadwala$^2$ \and Sheela Verma$^*$}

\address{$1$ Department of Mathematical Sciences, Indian Institute of Technology (BHU) Varanasi, India.}
\email{sagarbasak.rs.mat22@itbhu.ac.in}

\address{$2$ Indian Institute of Science Education and Research, Pune, India.}
\email{anisa@iiserpune.ac.in}

\address{$*$  Corresponding author, Department of Mathematical Sciences, Indian Institute of Technology (BHU), Varanasi, India}
\email{sheela.mat@iitbhu.ac.in}

\DeclareUnicodeCharacter{2212}{-}
\begin{document}

\begin{abstract}
We consider mixed Steklov-Dirichlet eigenvalue problem on smooth bounded domains in Riemannian manifolds. Under certain symmetry assumptions on multiconnected domains in $\mathbb{R}^{n}$ with a spherical hole, we obtain isoperimetric inequalities for $k$-th Steklov-Dirichlet eigenvalues for $2 \leq k \leq n+1$. We extend Theorem 3.1 of \cite{gavitone2023isoperimetric} from Euclidean domains to domains in space forms, that is, we obtain sharp lower and upper bounds of the first Steklov-Dirichlet eigenvalue on bounded star-shaped domains in the unit $n$-sphere and in the hyperbolic space. 
\end{abstract}

\maketitle

\section{Introduction}
Let $(\tilde{M}, \tilde{g})$ be a Riemannian manifold and let $D$ denote the Levi–Civita connection of $(\tilde{M}, \tilde{g})$. For a smooth vector field $X$ on $\tilde{M}$ the divergence $div(X)$ is defined as $trace(DX)$.
For a smooth function $f : \tilde{M} \rightarrow \mathbb{R}$, the gradient $\nabla f$ is defined by $\tilde{g}( \nabla f (p), v) = df (p)(v)~ (p \in \tilde{M}, v \in T_p\tilde{M})$ and the Laplace–Beltrami operator $\Delta$ is defined by $\Delta f = div(\nabla f )$. 
Throughout this paper, $dV$
denotes the volume element of $(\tilde{M}, \tilde{g})$.

 The Riemannian manifolds we consider in this article are going to be space forms, that is, the complete simply connected Riemannian manifolds with constant sectional curvature. They are (i) the Euclidean space $\mathbb{E}^n$, (ii) the unit sphere $S^n = \{(x_1, x_2,... ,x_{n+1}) \in \mathbb{R}^{n+1} \, | \, \sum_{i=1}^{n+1} x_i^2 = 1\}$ with induced Riemannian metric $<  , >$ from the Euclidean space $\mathbb{R}^{n+1}$, and (iii) the hyperbolic space $\mathbb{H}^n := \{(x_1, x_2, \ldots ,x_{n+1}) \in \mathbb{R}^{n+1} \,| \,
\sum_{i=1}^n x_i^2 - x_{n+1}^2 = −1 ~\mbox{ and }~ x_{n+1} > 0\}$ with
the Riemannian metric induced from the quadratic form $(x, y) := \sum_{i=1}^n x_i\, y_i - x_{n+1}\, y_{n+1}, ~\mbox{
where } x = (x_1, x_2,... ,x_{n+1}) \mbox{ and } y = (y_1, y_2,... ,y_{n+1})$.

Finding sharp bounds for Steklov-Dirichlet eigenvalues among a constrained family of doubly connected domains is a classical problem which has caught attention of many authors \cite{hong2023first, michetti2022steklov} owing to its physical significance. In this article, we consider mixed Steklov-Dirichlet problem on domains with a spherical hole. Let  $\tilde{\Omega}$ 
be a bounded domain in the Riemannian manifold $
\Tilde{
M
}$
with smooth boundary $\partial \tilde{
\Omega
}
$. Let $B_{r} \subset 
\Tilde{
M
}
$ be a ball of radius ${r}$ such that $\overline{B_{r}} \subset 
\tilde{
\Omega
}
$. Consider the following Steklov-Dirichlet eigenvalue problem on $\Omega = \tilde{\Omega} \setminus \overline{B_{r}}$.
\begin{align} \label{eqn: SD problem}
\begin{cases}
        \Delta u=0 & \text{in} \,\, \Omega,\\
        u=0 & \text{on}\,\, \partial B_{r},\\
        \frac{\partial u}{\partial \nu }=\sigma u & \text{on} \,\, \partial \tilde{\Omega}  
    \end{cases}    
\end{align}
It is known \cite{agranovich2006mixed} that this problem has discrete spectrum 
\begin{align*}
    0 < \sigma_1 \leq \sigma_2 \leq \sigma_3 \leq \cdots \nearrow \infty,
\end{align*}
counted with multiplicity.

Various bounds for the first Steklov-Dirichlet eigenvalue relating the geometry of the underlying domains have been obtained for bounded domains in $\mathbb{R}^{2}$. Using the theory of conformal mapping, Hersch and Payne \cite{hersch1968extremal} obtained sharp upper bounds for the first Steklov-Dirichlet eigenvalue on planar doubly connected bounded domains. In \cite{dittmar2000mixed}, using circular symmetrization and a distortion theorem for conformal mappings of an annulus, lower bounds for the first Steklov-Dirichlet eigenvalue were derived for doubly connected planar domains. For bounded domains in $\mathbb{R}^{n}$, Banuelos et al. \cite{banuelos2010eigenvalue} obtained some classical inequalities comparing eigenvalues of the mixed Steklov–Dirichlet problem and mixed Steklov–Neumann problem.

Many isoperimetric bounds and monotonicity results for the first Steklov-Dirichlet eigenvalue have been derived for various domains in $\mathbb{R}^{n}$ \cite{gavitone2023monotonicity,hong2022first}. For instance, in \cite{verma2020eigenvalue} it was proved that, if $n \geq 3$ and $\tilde{\Omega}$ is a ball of fixed radius, the first Steklov-Dirichlet eigenvalue $\sigma_{1}(\Omega)$ is maximum when $\tilde{\Omega}$ and $B_{r}$ are concentric. This result was later proved for planar annular domains and for domains contained in two-point homogeneous spaces \cite{ftouhi2022place,seo2021shape}. Generalising this, Paoli et al \cite{paoli2021stability} proved that concentric annular domain locally minimizes first Steklov-Dirichlet eigenvalue when $\tilde{\Omega}$ varies over nearly spherical sets of fixed volume. Further, if $\tilde{\Omega}$ is an open, bounded and convex set contained in a suitable ball, Gavitone et al \cite{gavitone2023isoperimetric} proved that, under a volume constraint, $\sigma_{1}(\Omega)$ is maximum for concentric annular domain. In particular, they proved the following result.

\begin{theorem} [Theorem 1.1, \cite{gavitone2023isoperimetric}] \label{Thm:upper bound in convex domains}
 Fix $n \geq 2$. Let $\tilde{\Omega}$ be a bounded convex open set in $\mathbb{R}^{n}$. Let $B_r$ be a ball in $\mathbb{R}^{n}$ of radius $r$ such that $\overline{B_r} \subset \tilde{\Omega} \subset B_{\Bar{r}}$ where $B_{\Bar{r}} $ is the ball in $\mathbb{R}^{n}$, centered at the origin, with radius $\Bar{r}$ defined, in terms of $r$ and $n$, as
 \begin{equation*}
     \Bar{r} = 
     \begin{cases}
     \begin{aligned}
& r e^{\sqrt{2}} & \text{ if }  n=2,\\
& r \left[ \frac{(n-1) + (n-2) \sqrt{2(n-1)}}{n-1} \right]^{\frac{1}{n-2}} & \text{ if }  n \geq 3.
 \end{aligned}
     \end{cases}
 \end{equation*}
Then, 
$\sigma_1 (\tilde{\Omega} \backslash \overline{B_r}) \leq \sigma_1 (B_{R,r})$,  
where $B_{R,r}$ is the concentric annular domain with inner radius $r$ and outer radius $R$ such that Vol$(\tilde{\Omega} \backslash \overline{B_r})$ = Vol$(B_{R,r})$.
\end{theorem}

In \cite{gavitone2023isoperimetric}, the authors have also obtained sharp lower and upper bounds for $\sigma_{1}(\Omega)$ in terms of the minimal and maximal distances between the center of $\tilde{\Omega}$ and the outer boundary of $\tilde{\Omega}$.

In this article, we generalise the above mentioned results in two directions: (1) To find sharp lower and upper bounds for doubly connected star shaped domains in non-Euclidean space forms, 
(2) To find optimal shapes for higher Steklov-Dirichlet eigenvalues (Theorem \ref{Thm:higher eigenvalues}) under a volume constraint and with certain symmetry assumptions on Euclidean domains. 

This article is organised as follows: In Section \ref{sec:star shaped in constant curvature}, we discuss the geometry of star shaped domains in non-Euclidean space forms.
We then derive sharp bounds for the first Steklov-Dirichlet eigenvalue on doubly connected star shaped domains with a spherical hole. In section \ref{sec: SD on annular domain}, we study the Steklov-Dirichlet eigenvalues, the corresponding eigenfunctions and their behaviour in annular domains. Certain integral inequalities on domains with given symmetry assumptions are proved in Section \ref{sec:integral inequalities}. In Section \ref{sec: higher eigenvalues}, we prove isoperimetric bounds for higher eigenvalues. Finally, in Section \ref{sec:remark on star shaped}, we state a generalisation of Theorem \ref{Thm:upper bound in convex domains} for star shaped domains (Theorem \ref{Thm:upper bound}). This theorem can be proved following the ideas in \cite{gavitone2023isoperimetric}. In this section, we also mention some extensions of our results in different directions for future work. 
\section{Sharp bounds for the first Steklov-Dirichlet eigenvalue on star shaped domains} \label{sec:star shaped in constant curvature}
Let $(\tilde{M},\tilde{g})$ be a complete Riemannian manifold and let $
\tilde{\Omega} 
\subset 
\Tilde{M}
$ be a bounded domain having smooth boundary $\partial \tilde{\Omega}$. We further assume that $
\tilde{\Omega}$ is a star shaped domain with respect to 
a point $p \in 
\tilde{
\Omega
}
$.  Then, for every $u \in T_{p}
\Tilde{
M
}
$, there exists a unique point $q \in \partial \tilde{
\Omega
}
$ such that $q = \exp_{p}(r_u u)$ for some $r_u > 0$. Thus $
\tilde{
\Omega
}
$ and $\partial \tilde{
\Omega}
$ can be represented as
\begin{equation*}
\partial 
\tilde{
\Omega
} 
= \{ \exp_{p}(r_u u) \, | \, u \in T_{p}\Tilde{M}, \norm{u} = 1 \},
\end{equation*}
\begin{equation*}
\tilde{
\Omega
}
= \{ \exp_{p}(t u) \, | \, u \in T_{p}
\Tilde{
M
}
, \norm{u} = 1, 0 \leq t < r_u \}.
\end{equation*}
Define $\displaystyle R_m := \min_{u} r_u$ and $\displaystyle R_M := \max_{u} r_u$.

Let $\partial_r$ denote the radial vector field emanating from $p$, and let $\nu$ denote the unit outward normal to $\partial\tilde{\Omega}$. For any point $q \in \partial \tilde{\Omega}$, define $\cos \theta_u := \cos (\theta(q)) = \left\langle \nu(q), \partial_r(q) \right\rangle_{\tilde{g}}$. Since $\tilde{\Omega}$ is a star-shaped bounded domain, $\cos (\theta(q)) > 0$ and therefore $\theta(q) < \frac{\pi}{2}$ for all $q \in \partial\tilde{\Omega}$. By the compactness assumption of $\partial \tilde{\Omega}$, 
\begin{equation}\label{a}
\mbox{there exists }\; \alpha ~\mbox{ such that } 0 \leq \theta_u =\theta(q) \leq \alpha  < \frac{\pi}{2} \;  \mbox{ for all }~\; q \in \partial\tilde{\Omega} . ~ \mbox{ Let } ~ a : = \tan^{2} \alpha.
\end{equation}
If $\Tilde{M}$ is a space form then 
$\displaystyle \tan^{2} (\theta(q))= \frac{\|\overline{\nabla} r_{u}\|^2}{\sin_{\Tilde{M}}^{2}(r_{u})}\, \forall \, q \in \partial\Omega$. Here, $\overline{\nabla}r_u$ represents the tangential component of ${\nabla}r_u$, the gradient of $r_u$, and 
\begin{align*}
\sin_{\Tilde{M}}r :=
\begin{cases}
\sin r & \text{ for } r \in [0, \pi] \quad \text{ when }\Tilde{M} 
= 
S
^{n},  \\
r & \text{ for } r \in [0, \infty) \quad \text{ when }
\Tilde{
M} 
= \mathbb{R}^{n},  \\
\sinh r & \text{ for } r \in [0, \infty) \quad \text{ when }
M
= \mathbb{H}^{n}.
\end{cases}
\end{align*} 
Let $U_{p}\Tilde{
M
} 
:= \{u \in T_{p}\Tilde{
M
}
\, | \, \norm{u} = 1 \}$. For $p \in \tilde{M}$, and $r  > 0$, let $B_{r}(p)$ 
 denote the open geodesic ball in $\tilde{M}$ with center $p$ and radius ${r}$. When there is no confusion about the center $p$, we can denote $B_r(p)$ by $B_r$. For $\tilde{M}=S^n$, we take $r \in (0, \pi
 ]$. Clearly, for any $p \in S^2$, $B_{\frac{\pi}{2}}(p)$ is nothing but the open hemisphere of $S^n$ centered at $p$.
\begin{theorem} \label{Thm:lower bound}
Let $\tilde{M
 } 
 = \mathbb{R}^{n}, S^{n}$ or $\mathbb{H}^{n}$.
 Let $p \in \tilde{M}$, $r_1>0$ and $\tilde{\Omega} \subset \tilde{M}$ 
  be such that $\overline{B_{r_1}(p)} \subset \tilde{
 \Omega
 }
 $, and $\tilde{
 \Omega
 }
 $ is a star shaped domain with respect to $p$. 
 If $\Tilde{
 M
 } 
 = 
 S
 ^{n}$, we further assume that $\tilde{
 \Omega
 }
 \subset B_{\frac{\pi}{2}}(p)$. Then for $\Omega:= \tilde{\Omega} \backslash \overline{B_{r_1}}$, we have the following bounds for $\sigma_1 (\Omega)$ in terms of $a, R_M, R_m$, $\sigma_1 ( B_{R_m} \backslash \overline{B_{r_1}})$ and $\sigma_1 ( B_{R_M} \backslash \overline{B_{r_1}})$,
\begin{equation*}
     \begin{aligned}
 \left( \frac{1}{\sqrt{1+a}}\right) \left(\frac{\sin_{\Tilde{M}}^{n-1}(R_m)}{\sin_{\Tilde{M}}^{n-1}(R_M)}\right)  \sigma_1 ( B_{R_m} \backslash \overline{B_{r_1}}) \leq  \sigma_1 (\Omega) \leq  \left(\frac{\sin_{\Tilde{M}}^{n-1}\left(R_{M}\right)}{\sin_{\Tilde{M}}^{n-1}\left(R_{m}\right)} \right)  \sigma_1 ( B_{R_M} \backslash \overline{B_{r_1}}),  
     \end{aligned}
\end{equation*}
where $B_{R_M}$ and $B_{R_m}$ are balls centered at $p$ with radii $R_M$ and $R_m$, respectively. 
\end{theorem}

\begin{proof}
Let $f$ be a positive eigenfunction of (\ref{eqn: SD problem}) corresponding to $\sigma_1 (\Omega)$. Then $\|\nabla f\|^2 = \left(\frac{\partial f}{\partial r} \right)^2 + \frac{1}{\sin_{\Tilde{M}}^{2}r }\|\overline{\nabla} f\|^2$.
Therefore,
\begin{align*}
\int_\Omega{\|\nabla f\|^2}\, dV & = \int_{U_{p}\Tilde{M}} \int_{r_1}^{r_{u}} \left[ \left(\frac{\partial f}{\partial r} \right)^2 + \frac{1}{\sin_{\Tilde{M}}^{2}r }\|\overline{\nabla} f\|^2\right] \sin_{\Tilde{M}}^{n-1}r \, dr\, du \\
& \geq \int_{U_{p}\Tilde{M}} \int_{r_1}^{R_{m}} \left[ \left(\frac{\partial f}{\partial r} \right)^2 + \frac{1}{\sin_{\Tilde{M}}^{2}r }\|\overline{\nabla} f\|^2\right] \sin_{\Tilde{M}}^{n-1}r \, dr\, du = \int_{B_{R_m} \backslash \overline{B_{r_1}}}{\|\nabla f\|^2}\, dV,
\end{align*}
and
\begin{align*}
&~~~~~ \int_{\partial\Omega}{f^2}\, dS = \int_{u \in U_{p}\Tilde{M}} f^2 \, \sec (\theta_u)\, \sin_{\Tilde{M}}^{n-1}\left(r_{u}\right)\, du  & \leq \sec (\alpha) \int_{U_{p}\Tilde{M}} f^2 \, \sin_{\Tilde{M}}^{n-1}\left(r_{u}\right)\, du ~~~~~~~~~~~~~~\\
& \leq  \frac{\sqrt{1+a} \ \sin_{\Tilde{M}}^{n-1}\left(R_{M}\right)}{\sin_{\Tilde{M}}^{n-1}\left(R_{m}\right)} \int_{U_{p}\Tilde{M}} f^2 \, \sin_{\Tilde{M}}^{n-1}\left(R_{m}\right)\, du& = \frac{\sqrt{1+a} \ \sin_{\Tilde{M}}^{n-1}\left(R_{M}\right)}{\sin_{\Tilde{M}}^{n-1}\left(R_{m}\right)} \int_{\partial (B_{R_m} \backslash \overline{B_{r_1}})}{f^2}\, dS.
\end{align*}
Thus,
\begin{align*}
\frac{\sin_{\Tilde{M}}^{n-1}\left(R_{m}\right)}{\sqrt{1+a} \ \sin_{\Tilde{M}}^{n-1}\left(R_{M}\right)} \sigma_1 (B_{R_m} \backslash \overline{B_{r_1}}) \leq \frac{\sin_{\Tilde{M}}^{n-1}\left(R_{m}\right)}{\sqrt{1+a} \ \sin_{\Tilde{M}}^{n-1}\left(R_{M}\right)} \frac{\int_{B_{R_m} \backslash \overline{B_{r_1}}}{\|\nabla f\|^2}\, dV}{\int_{\partial (B_{R_m} \backslash \overline{B_{r_1}})}{f^2}\, dS} \leq  \frac{\displaystyle \int_\Omega{\|\nabla f\|^2}\, dV }{\displaystyle \int_{\partial\Omega}{f^2}\, dS} = \sigma_1 (\Omega).
\end{align*}

Similarly, let $h$ be the positive eigenfunction corresponding to $\sigma_1 (B_{R_M} \backslash \overline{B_{r_1}})$. Then 
\begin{align*}
\int_\Omega{\|\nabla h\|^2}\, dV & = \int_{U_{p}\Tilde{M}} \int_{r_1}^{r_{u}} \left[ \left(\frac{\partial h}{\partial r} \right)^2 + \frac{1}{\sin_{\Tilde{M}}^{2}r }\|\overline{\nabla} h\|^2\right] \sin_{\Tilde{M}}^{n-1}r \, dr\, du \\
& \leq \int_{U_{p}\Tilde{M}} \int_{r_1}^{R_{M}} \left[ \left(\frac{\partial h}{\partial r} \right)^2 + \frac{1}{\sin_{\Tilde{M}}^{2}r }\|\overline{\nabla} h\|^2\right] \sin_{\Tilde{M}}^{n-1}r \, dr\, du = \int_{B_{R_M} \backslash \overline{B_{r_1}}}{\|\nabla h\|^2}\, dV,
\end{align*}
and,
\begin{align*}
\int_{\partial\Omega}{h^2}\, dS & = \int_{U_{p}\Tilde{M}} h^2 \, \sec (\theta_u)\, \sin_{\Tilde{M}}^{n-1}\left(r_{u}\right) \, du \geq \int_{U_{p}\Tilde{M}} h^2 \, \sin_{\Tilde{M}}^{n-1}\left(r_{u}\right) \, du \\
& \geq   \frac{\sin_{\Tilde{M}}^{n-1}\left(R_{m}\right)}{\sin_{\Tilde{M}}^{n-1}\left(R_{M}\right)} \int_{U_{p}\Tilde{M}} h^2 \, \sin_{\Tilde{M}}^{n-1}\left(R_{M}\right)\, du \\
& =  \frac{\sin_{\Tilde{M}}^{n-1}\left(R_{m}\right)}{\sin_{\Tilde{M}}^{n-1}\left(R_{M}\right)} \int_{\partial (B_{R_M} \backslash \overline{B_{r_1}})}{h^2}\, dS.
\end{align*}
Thus,
\begin{align*}
 \sigma_1 (\Omega) \leq \frac{\displaystyle \int_\Omega{\|\nabla h\|^2}\, dV }{\displaystyle \int_{\partial\Omega}{h^2}\, dS} \leq \frac{\sin_{\Tilde{M}}^{n-1}\left(R_{M}\right)}{\sin_{\Tilde{M}}^{n-1}\left(R_{m}\right)} \frac{\int_{B_{R_M} \backslash \overline{B_{r_1}}}{\|\nabla h\|^2}\, dV}{\int_{\partial (B_{R_M} \backslash \overline{B_{r_1}})}{h^2}\, dS} = \frac{\sin_{\Tilde{M}}^{n-1}\left(R_{M}\right)}{\sin_{\Tilde{M}}^{n-1}\left(R_{m}\right)}  \sigma_1 ( B_{R_M} \backslash \overline{B_{r_1}}).
\end{align*}
\end{proof}

\section{The Steklov-Dirichlet problem on concentric annular domains} \label{sec: SD on annular domain}
In this section, we study some properties of the mixed Steklov-Dirichlet eigenvalues and the corresponding eigenfunctions on concentric annular domais in $\mathbb{R}^{n}$. We first use separation of variables method in order to compute the eigenfunctions, where in we need the following proposition which describes the eigenvalues and the eigenfunctions of $\Delta_{
S
^{n-1}}$ \cite[Sections 22.3, 22.4]{shubin1987pseudodifferential}. 

\begin{proposition} \label{multiplicity of eigenvalue}
		The set of all eigenvalues of $\Delta_{
  S
  ^{n-1}}$ is $\{l\,(l+n-2) : l \in \mathbb{N} \cup \{0\} \}$. The eigenfunctions corresponding to each eigenvalue $l(l+n-2)$ are the spherical harmonics of degree $l$ and thus, the multiplicity of the eigenvalue $l(l+n-2)$ equals the dimension of $\mathcal{H}_{l}$, the space of harmonic homogeneous polynomials of degree $l$ on $\mathbb{R}^{n}$.
\end{proposition}

Let $\Omega_0 := B_{R_2} \backslash \overline{B_{R_1}}$, where $B_{R_1}$ and $B_{R_2}$ are concentric balls in $\mathbb{R}^{n}$ of radius $R_1$ and $R_2$ respectively. Without loss of generality we assume that both these balls are centered at the origin. Now, we find the eigenvalues and the eigenfunctions of the following Steklov-Dirichlet eigenvalue problem on $\Omega_0$ 
and study some of their properties. 
\begin{equation}
    \begin{cases}
        \Delta u=0 & \text{in} \,\, \Omega_0,\\
        u=0 & \text{on}\,\, \partial B_{R_1},\\
        \frac{\partial u}{\partial \nu }=\sigma u & \text{on} \,\, \partial B_{R_2.}  \label{eq:Steklov-Dirichlet(ball)}
    \end{cases}
\end{equation}

Let $u(r,\omega)=f(r)g(\omega)$ be a smooth function, where $f$ is a radial function defined on $[R_1, R_2]$ and $g$ is an eigenfunction of $\Delta_{
S
^{n-1}}$ corresponding to the eigenvalue $
l(l+n-2)$. 
Now,
\begin{align*}
     \Delta u(r, \omega) &= g(\omega)\left(-f''(r)- \frac{n-1}{r} f'(r) \right) + \frac{f(r)}{r^2} \Delta_{
     S
     ^{n-1}} g(\omega)\\
     &= g(\omega) \left(-f''(r)- \frac{n-1}{r} f'(r) + \frac{f(r)}{r^2} l(l+n-2) \right).
\end{align*}
If $u$ is a solution of (\ref{eq:Steklov-Dirichlet(ball)})  then the function $f$ satisfies 
\begin{eqnarray} \label{ODE for f}
    \begin{array}{ll}
        -f''(r)-\frac{n-1}{r}f'(r)+ \frac{l(l+n-2)}{r^2}f(r)=0 ~\mbox{ for } ~ r \in (R_1,R_2), \\
    f(R_1)=0, \,\, f'(R_2)=\sigma f(R_2).
    \end{array}
\end{eqnarray}
We know that the eigenfunctions of \eqref{ODE for f} are given by
\begin{equation*}
    f_{0}(r)= \begin{cases}
    \frac{1}{R_1^{n-2}} - \frac{1}{r^{n-2}} \,\,\, & \text{if} \,\, n > 2 \\
    ln(r) - ln(R_1) \,\,\, & \text{if} \,\, n=2,
\end{cases} 
\end{equation*}
and,
\begin{equation} \label{eqn: eigenfuction}
    f_{l}(r)= r^l - \frac{R_1^{n+2l-2}}{r^{n+l-2}},\,\,\, \text{for}\,\, l\geq 1
 \end{equation}
 corresponding to the eigenvalues
\begin{equation*}
    \sigma_{(0)}(\Omega_0)= \begin{cases}
    \frac{(n-2)R_{1}^{n-2}}{R_{2}^{n-1} - R_{2} R_{1}^{n-2}} \,\,\, & \text{for} \,\, n > 2 \\
    \frac{1}{R_{2} (ln(R_2) - ln(R_1))} \,\,\, & \text{for} \,\, n=2,
\end{cases} 
\end{equation*}
and,
\begin{align*}
    \sigma_{(l)}(\Omega_0)= \frac{l  R_2^{n+2l-2} + (n+l-2) R_1^{n+2l-2}}{ R_2^{n+2l-1} - R_2 R_1^{n+2l-2}}, \,\, \text{for}\,\, l\geq 1,
    \end{align*}
respectively. Here, $\sigma_{(i)}(\Omega_0)$, $i \in \mathbb{N} \cup \{0\}$ denotes the $i+1$-th Steklov Dirichlet eigenvalue on $\Omega_0$ counted without multiplicity.
\begin{remark}
   For $2 \leq i \leq n+1$, $\sigma_i(\Omega_0) = \sigma_2(\Omega_0) = \sigma_{(1)}(\Omega_0)$ i.e., $\sigma_{(1)}(\Omega_0)$ has multiplicity $n$ and the corresponding eigenfunctions are $f_{1}(r) \frac{x_i}{r}, i = 1, 2, \ldots n$. Further, note that 
   \begin{align*}
       \sigma_2(\Omega_0) = \sigma_{(1)}(\Omega_0) = \frac{R_2^{n} + (n-1) R_1^{n}}{ R_2^{n+1} - R_2 R_1^{n}} = \frac{\displaystyle \int_{\Omega_{0}} \left( (f'_1(r))^2 + \frac{(n-1)}{r^2} f_1^2(r)\right) \, dV }{\displaystyle \int_{\partial B_{R_2} }f_1^2(r) \, dS}. 
  \end{align*}
\end{remark}

\begin{lemma} \label{increasing lemma}
    Let $f_1:[R_1,\infty) \rightarrow \mathbb{R} $ be as defined in \eqref{eqn: eigenfuction}. Define $F,G : [R_1,\infty) \rightarrow \mathbb{R}$ as $ F(r) := \left((f'_1(r))^2 + \frac{(n-1)}{r^2} f_1^2(r)\right)$, and $ G(r) := \left( 2f_1(r)f_1'(r) + \frac{n-1}{r}f_1^2(r) \right)$. Then
       $F$ is a decreasing function of $r$ and
       $G$ is an increasing function of $r$.
\end{lemma}
    \begin{proof}
        We have $f_1(r)= r-\frac{R_1^n}{r^{n-1}}$. Therefore, $f_1'(r)= 1 + (n-1)\left(\frac{R_1}{r}\right)^n.$ Substituting these in the definition of $F$ and $G$ we get,
$$F(r) = 
n \left( 1+ (n-1) \left(\frac{R_1}{r}\right)^{2n}\right), ~~~
 G(r)= 
(n+1)r - 2\frac{R_1^n}{r^{n-1}} - (n-1)\frac{R_1^{2n}}{r^{2n-1}}.$$ Consequently, 
$$F'(r) = -2n(n-1)\frac{R_1^{2n}}{r^{2n+1}} < 0, ~~\mbox{ and }~~G'(r) = (n+1) + 2(n-1)\frac{R_1^n}{r^n}+(2n-1)(n-1)\frac{R_1^{2n}}{r^{2n}} > 0.$$ Hence proved.
           
            
    \end{proof}

\section{Integral inequalities related to Steklov-Dirichlet eigenfunctions on Annular domains} \label{sec:integral inequalities}
We first define Euclidean domains with symmetry of order $s$, which appears in the statement of our main result.

For $s\in \mathbb{N},$ and $i,j\in \left\{1,2,\ldots, n\right\}$, let $T_{i,j}^{\frac{2 \pi}{s}}$ denote the rotation, in the anti-clockwise direction, around the origin by an angle $\frac{2 \pi}{s}$ in the coordinate plane $(x_{i}, x_{j})$. 
\begin{definition}
A domain $\Omega \subset \mathbb{R}^{n}$ is said to be \textit{symmetric of order $s$} with respect to the origin, if there exists a rotation $\rho$ of $\mathbb{R}^n$ such that  $T_{i,j}^{\frac{2 \pi}{s}}(\rho(\Omega)) = \rho(\Omega), \text{ for all } \, i,j \in \left\{1,2,\ldots, n\right\}.$ 
\end{definition}
\begin{definition}
A domain $\Omega \subset \mathbb{R}^{n}$ is said to be \textit{centrally symmetric} with respect to the origin, if $-x \in \Omega$ whenever $x \in \Omega$.
\end{definition}

Let $E_1, E_2, \ldots, E_n$ be the standard orthonormal basis of $\mathbb{R}^{n}$ and $(x_1, x_2, \ldots, x_n)$ be the standard normal coordinate system on $\mathbb{R}^{n}$. Let $\tilde{\Omega}$  be a 
smooth bounded domain in $\mathbb{R}^n$. Let $B_{R_1}$ be a ball of radius $R_1$ such that $B_{R_1} \subset \tilde{\Omega}$. Without loss of generality, we can assume that, $\tilde{\Omega}$ contains the origin of $\mathbb{R}^n$ and that $B_{R_1}$ is centered at the origin. Let $B_{R_2}$ be a ball of radius $R_2$ centered at the origin such that Vol$(B_{R_2})$ = Vol$(\tilde \Omega)$. Clearly, $R_2 > R_1$. Define $\Omega:= \tilde{\Omega} \backslash \overline{B_{R_1}}$ and 
$\Omega_0:= B_{R_2} \backslash \overline{B_{R_1}}$.

The following proposition is useful in finding test functions for the variational charecterization of the Steklov-Dirichlet eigenvalues on $\Omega$.
\begin{proposition} \label{prop:integral expression}
 Let $g : (0, \infty) \rightarrow \mathbb{R}$ be a smooth function. Let $\Omega$  be a bounded smooth domain in $\mathbb{R}^n$.
 \begin{enumerate}
     \item If $\Omega$ 
     has a symmetry of order 2, then for each $i=1,2,\ldots n$, we have\\[2mm]
        (a) $\displaystyle \int_{x \in \Omega}   g(\|x\|) x_i\, dV = 0$,  \label{eqn:integral 1}~~~
       (b) $\displaystyle \int_{x \in \partial \Omega}   g(\|x\|) x_i\, dS = 0$. 
        \label{eqn:integral 3}

\item If $\Omega$ 
has symmetry of order 4, then for each $i, j=1,2,\dots n$, $i\neq j$,  we have\\[2mm]
       (a) $\displaystyle \int_{ x \in \Omega} g(\|x\|) x_i x_j \, dV = 0$, \label{eqn:integral 2}
        $\displaystyle \int_{x \in \partial \Omega}  g(\|x\|) x_i x_j \,dS = 0$.
        \label{eqn:integral 4}
 \end{enumerate}
\end{proposition}
\begin{proof}
 \begin{enumerate} 
        \item $\Omega$ has symmetry of order $2$. Therefore, if we take the transformation $x=R_{i,j}^\frac{2\pi}{2}(y) = -y$ for $j \neq i$, then we get
\begin{enumerate}
    \item $\displaystyle \int_{x \in \Omega}   g(\|x\|) x_i\, dV = -\displaystyle \int_{y \in \Omega}   g(\|y\|) y_i\, dV $. Thus, $\displaystyle \int_{x \in \Omega}   g(\|x\|) x_i\, dV = 0,$ for $i=1,2,\ldots n$. 
    \item  Since $\partial \Omega$ is also symmetric of order $2$, proof follows as in \eqref{eqn:integral 1} (a).
\end{enumerate}
\item Here, $\Omega$ has symmetry of order $4$. So, if we take the transformation $x=R_{i,j}^\frac{2\pi}{4}(y)$ we get \\
\begin{enumerate}
    \item  $\displaystyle \int_{x \in \Omega} g(\|x\|) x_i x_j \, dV = -\displaystyle \int_{y\in \Omega} g(\|y\|) y_i y_j \, dV$. Thus, $\displaystyle \int_\Omega g(\|x\|) x_i x_j \, dV = 0$.
    \item The proof is similar to that of \eqref{eqn:integral 2} (a).  
\end{enumerate}
 \end{enumerate}   
\end{proof}
Observe that for the function $f_1$ defined in \eqref{eqn: eigenfuction},
\begin{align*} 
\Bigg \langle  \nabla \left( \frac{f_1(\|x\|)}{\|x\|}x_i \right), E_j \Bigg \rangle = \frac{\partial}{\partial x_j} \left( \frac{f_1(\|x\|)}{\|x\|}x_i \right) = 
\begin{cases}
   \frac{f_1'(\|x\|)}{\|x\|^2} x_j \ x_i- \frac{f_1(\|x\|)}{\|x\|^3} x_j \ x_i, &  \text{ for } j \neq i, \\[2mm]
    \frac{f_1'(\|x\|)}{\|x\|^2} x_i^2+ \frac{f_1(\|x\|)}{\|x\|^3}(\|x\|^2-x_i^2), &  \text{ for } j = i.
\end{cases}
\end{align*}
Using this, we get 
\begin{align} \label{eq: inner product of gradient1}
\Bigg \langle  \nabla \left( \frac{f_1(\|x\|)}{\|x\|}x_i \right), \nabla \left( \frac{f_1(\|x\|)}{\|x\|}x_j \right) \Bigg \rangle =  
\begin{cases}
   \left( \frac{(f_1'(\|x\|))^2}{\|x\|^2} - \frac{(f_1(\|x\|))^2}{\|x\|^4}\right)x_ix_j, &  \text{ for } j \neq i, \\[2mm]
    \left( \frac{(f_1'(\|x\|))^2}{\|x\|^2}x_i^2 - \frac{f_1^2(\|x\|)}{\|x\|^4}x_i^2  + \frac{f_1^2(\|x\|)}{\|x\|^2}\right), &  \text{ for } j = i.\\
\end{cases}
\end{align}
\begin{align} \label{eq: inner product of gradient2}
  \Bigg \langle \nabla\left(f_1(\|x\|)\right),\nabla \left(\frac{f_1(\|x\|)}{\|x\|}x_j\right)  \Bigg \rangle = \left(\frac{\left(f_1'(\|x\|)\right)^2}{\|x\|} - \frac{f_1'(\|x\|) f_1(\|x\|)}{\|x\|^2}\right)x_i.  
\end{align}
From equation \eqref{eq: inner product of gradient1}, \eqref{eq: inner product of gradient2} and Proposition \ref{prop:integral expression} we conclude that
\begin{cor} \label{cor:integral expression}
    If $\Omega 
    $ is a 
    bounded smooth domain in $\mathbb{R}^n$ having symmetry of order 4 and if $f_1$ is as defined in \eqref{eqn: eigenfuction}. Then, for each $i,j=1,2,\dots n$, $i\neq j$ we have \\
    \begin{enumerate} 
        \item  $\displaystyle \int_{x \in \partial \Omega}  f_1(\|x\|) \frac{f_1(\|x\|)}{\|x\|}x_i\, dS = 0,$ \label{proof (i)}\\
         \item $\displaystyle \int_{x \in \partial \Omega}  \frac{f_1(\|x\|)}{\|x\|}x_i. \frac{f_1(\|x\|)}{\|x\|}x_j \,dS=0,$
        \item $\displaystyle \int_{ x \in \Omega} \left< \nabla f_1(\|x\|), \nabla \left(\frac{f_1(\|x\|)}{\|x\|}x_i \right)\right> dV=0,$ 
        \item $\displaystyle \int_{x \in \Omega} \left< \nabla \left(\frac{f_1(\|x\|)}{\|x\|} x_i \right) , \nabla \left(\frac{f_1(\|x\|)}{\|x\|}x_j \right)\right> dV=0$. 
    \end{enumerate}
\end{cor}
\begin{lemma}  \label{constant lemma}
     Let $\Omega$ be a 
     bounded smooth domain in $\mathbb{R}^n$ having symmetry of order $4$. Let $g: \mathbb{R}^n \rightarrow \mathbb{R}$ be a positive radial function. 
     Then, there a exist constant $A > 0$ such that 
     \begin{equation}
         \int_{x \in \Omega} g(\|x\|) x_i^2\,\, dV= A\,\, \text{for all}\, \,i\in \{1,2 \dots n\}.
     \end{equation}
\end{lemma}
     \begin{proof}
         Fixing an $i\neq 1$, and taking the transformation $x=R_{1,i}^\frac{2\pi}{4}(y)$, we obtain 
         \begin{equation*}
             \int_{x \in \Omega} g(\|x\|) x_i^2\,\, dV= \int_{y \in \Omega} g(\|y\|) y_1^2\,\, dV =:A.
         \end{equation*}
         Clearly, $A>0$. Hence, proved.
 \end{proof}
Recall that $\Omega:= \tilde{\Omega} \backslash \overline{B_{R_1}}$ and $\Omega_0:= B_{R_2} \backslash \overline{B_{R_1}}$.
 \begin{lemma} \label{lem: integral on domain}
    Let  $F:[R_1, \infty)\rightarrow \mathbb{R}$ be as defined in Lemma (\ref{increasing lemma}). 
    Then, the following inequality holds:
    \begin{equation}
       \int_\Omega F(\|x\|) dV \leq \int_{\Omega_0} F(\|x\|) dV.
    \end{equation} \label{lower Bound}
    \end{lemma}
    \begin{proof} The domain $\Omega$ is a  disjoint union of $(\Omega \cap \Omega_0)$ and $\Omega \backslash (\Omega \cap \Omega_0)$. Therefore,
        \begin{align*}
            \int_\Omega F(\|x\|) dV &  = \int_{\Omega \cap \Omega_0} F(\|x\|) dV + \int_{\Omega \backslash (\Omega \cap \Omega_0)} F(\|x\|) dV \\
             & =   \int_{\Omega_0} F(\|x\|) dV  - \int_{\Omega_0 \backslash (\Omega \cap \Omega_0)} F(\|x\|) dV + \int_{\Omega \backslash (\Omega \cap \Omega_0)} F(\|x\|) dV. 
        \end{align*}
        By lemma (\ref{increasing lemma}) we know that $F$ is a decreasing function of $r$. Then
        \begin{align}
            F(\|x\|) > F(R_2)\,\,\,\, \,\,\, \text{for} \,\,\, x \in \Omega_0 \backslash (\Omega \cap \Omega_0), \label{inequality 1} \\
            \text{and} \,\,\, F(\|x\|) < F(R_2)\,\,\,\, \,\,\, \text{for} \,\,\, x \in \Omega \backslash (\Omega \cap \Omega_0). \label{inequality 2}
        \end{align} 
        Therefore,
        \begin{align*}
            \int_{\Omega} F(\|x\|) dV \leq \int_{\Omega_0} F(\|x\|) dV  - \int_{\Omega_0 \backslash (\Omega \cap \Omega_0)} F(R_2) dV + \int_{\Omega \backslash (\Omega \cap \Omega_0)} F(R_2) dV. 
        \end{align*}
         Now, since Vol$(\Omega_0 \backslash (\Omega \cap \Omega_0))$ = Vol$(\Omega \backslash (\Omega \cap \Omega_0))$, we get the desired inequality. 
    \end{proof}
Let $\partial \tilde{\Omega}$ and $\partial{B_{R_{2}}}$ denote the boundaries of $\tilde{\Omega}$ and ${B_{R_{2}}},$ respectively. 

\begin{lemma} \label{lem:integral on boundary}
    For $f_1$, a function of $\|x\|$ as defined in  \eqref{eqn: eigenfuction}, the following inequality holds:
    \begin{align}
        \int_{x \in \partial \tilde{\Omega}} f_1^2(\|x\|) \ dS \geq \int_{x \in \partial B_{R_2}} f_1^2(\|x\|) \ dS.
    \end{align} \label{upper  bound lemma}
    \begin{proof}
    Recall that $
    S
    ^{n-1}$ is the $(n-1)$-dimensional unit sphere in $\mathbb{R}^n$ and that $R_u = \sup \{r\,|\,\, ru \in \partial \tilde{\Omega},
    u \in S^{n-1}\} $.
        Let $C= \{ R_u u \,| \,\, u\in 
        S
        ^{n-1}\}$. Then, 
        \begin{align*}
             \int_{x \in \partial \tilde{\Omega}} f_1^2(\|x\|) dS & \geq \int_{x \in C} f_1^2(\|x\|) dS \\
             &= \int_{u \in 
             S
             ^{n-1}} f_1^2(R_u) sec(\theta) R_u^{n-1} du  \\
             & \geq \int_{ u\in 
             S
             ^{n-1}} f_1^2(R_u) R_u^{n-1} du \\
             & = \int_{u \in 
             S
             ^{n-1}}\int_{r \in [R_1,R_u]} \left ( 2 f_1(r) f'_1(r) r^{n-1} + f_1^2(r) (n-1) r^{n-2} \right) dr du\\
             & = \int_{
             S
             ^{n-1}}\int_{R_1}^{R_u} \left( 2 f_1(r) f'_1(r)  + f_1^2(r) \frac{(n-1)}{r}  \right)r^{n-1} dr du\\
             & \geq \int_{x\in \Omega} \left( 2 f_1(\|x\|) f'_1(\|x\|)  + f_1^2(\|x\|) \frac{(n-1)}{\|x\|}  \right) dV \\
             &= \int_{x\in \Omega} G(\|x\|)dV.
        \end{align*}
        where $G(r) := \left( 2 f_1(r) f'_1(r)  + f_1^2(r) \frac{(n-1)}{r}  \right) $ is defined as in lemma (\ref{increasing lemma}). 
        Thus, we have \begin{align}
            \int_{x \in \partial \tilde{\Omega}} f_1^2(\|x\|) dS \geq \int_{x \in\Omega} G(\|x\|) dV. \label{boundary domain inequality}
        \end{align}
        Now, by lemma (\ref{increasing lemma}), $G$ is an increasing function of $r$. So,
        \begin{equation} \label{inequality 3}
            G(\|x\|) < G(R_2) \,\,\, \text{for}\,\, x \in \Omega_0 \backslash (\Omega \cap \Omega_0), 
           ~\mbox{ and } ~~ G(\|x\|) > G(R_2)\,\, \text{for} \,\,\, x \in \Omega \backslash (\Omega \cap \Omega_0).
       \end{equation}
        Now 
         \begin{align*}
            \int_ {x \in \Omega} G(\|x\|) dV & = \int_{x \in \Omega \cap \Omega_0} G(\|x\|) dV+ \int_{x \in \Omega \backslash (\Omega \cap \Omega_0)} G(\|x\|) dV\\
            & =  \int_{x \in \Omega_0} G(\|x\|) dV  - \int_{x \in \Omega_0 \backslash (\Omega \cap \Omega_0)} G(\|x\|) dV + \int_{x \in \Omega \backslash (\Omega \cap \Omega_0)} G(\|x\|) dV.
        \end{align*} 
        Thus, from inequality (\ref{inequality 3}), we get 
        \begin{align*}
            \int_{x \in \Omega} G(\|x\|) dV \geq \int_{x \in \Omega_0} G(\|x\|) dV  - \int_{x \in \Omega_0 \backslash (\Omega \cap \Omega_0)} G(R_2) dV + \int_{x \in \Omega \backslash (\Omega \cap \Omega_0)} G(R_2) dV.
        \end{align*}
         Since Vol$(\Omega_0 \backslash (\Omega \cap \Omega_0))$ = Vol$(\Omega \backslash (\Omega \cap \Omega_0))$, we get
         \begin{align*}
             \int_{x \in \Omega} G(\|x\|) dV & \geq  \int_{x \in \Omega_0} G(\|x\|) dV \\
              &= \int_{x \in \Omega_0}  \left( 2 f_1(\|x\|) f'_1(\|x\|)  + f_1^2(\|x\|) \frac{(n-1)}{\|x\|}  \right) dV \\
             & = \int_{u \in S^{n-1}} \int_{r=R_1}^{R_2} \left( 2 f_1(r) f'_1(r)  + f_1^2(r) \frac{(n-1)}{r}  \right) r^{n-1} dr du\\
             & = \int_{S^{n-1}} f_1^2(R_2) R_2^{n-1} du \\ 
             & = \int_{x \in \partial B_{R_2}} f_1^2(\|x\|) dS.
         \end{align*}
         Thus,
             $\displaystyle \int_{x \in\partial \tilde{\Omega}} f_1^2(\|x\|) dS \geq \int_{x \in \partial B_{R_2}} f_1^2(\|x\|) dS$.   \end{proof}
\end{lemma}
  
\section{Bounds for Higher eigenvalues} \label{sec: higher eigenvalues}
Recall from section \ref{sec:integral inequalities} that $\tilde{\Omega}$ is a 
smooth bounded domain in $\mathbb{R}^n$, $B_{R_1}$ is a ball of radius $R_1$ such that $B_{R_1} \subset \tilde{\Omega}$. Without loss of generality, we had assumed that, $\tilde{\Omega}$ contains the origin of $\mathbb{R}^n$ and that $B_{R_1}$ is centered at the origin. We also recall that $B_{R_2}$ is a ball of radius $R_2$ centered at the origin such that Vol$(B_{R_2})$ = Vol$(\tilde \Omega)$. Clearly, $R_2 > R_1$. Define $\Omega:= \tilde{\Omega} \backslash \overline{B_{R_1}}$ and 
$\Omega_0:= B_{R_2} \backslash \overline{B_{R_1}}$. Then, Vol$(\Omega) = $Vol$(\Omega_0)$.

We consider the following Steklov-Dirichlet eigenvalue problem on $\Omega$:
\begin{equation}
    \begin{cases}
        \Delta u=0 & \text{in} \,\, \Omega,\\
        u=0 & \text{on}\,\, \partial B_{R_1},\\
        \frac{\partial u}{\partial \nu }=\sigma u & \text{on} \,\, \partial \tilde{\Omega.}  \label{eq:Steklov-Dirichlet}
    \end{cases}
\end{equation}

For each $1 \leq k < \infty$,  
the $k$-th eigenvalue of (\ref{eq:Steklov-Dirichlet}), viz. $\sigma_{k}$, 
admits the following variational charactrization 
\begin{equation}  \label{charactrizaton}
    \sigma_{k}(\Omega):=\min_{E\in \mathcal{H}_{k, 0}(\Omega)} \max_{u(\neq 0) \in E} R(u),
\end{equation}
where $R(u):= \frac{\displaystyle \int_{\Omega} \| \nabla u \|^2 dv}{\displaystyle \int_{\partial\tilde{\Omega}} \|u\|^2 ds}$, and ${\mathcal{H}_{k,0}(\Omega)}$ is the collection of all the $k$-dimensional subspaces of the Sobolev space $\Tilde{H}_{0}^1  (\Omega):=\{ u \in H^1(\Omega) : u=0 \,\,\text{on} \,\, \partial B_{R_1}\}$.

We now state our main theorem:
\begin{theorem} \label{Thm:higher eigenvalues}
    Let $\tilde{\Omega}$ be a 
bounded smooth domain in $\mathbb{R}^n$ having symmetry of order 4 with respect to the origin. Let $\Omega=\tilde{\Omega} \backslash \overline{B_{R_1}}$ 
and let $\sigma_{k} $ be the $k${th}  eigenvalue of (\ref{eq:Steklov-Dirichlet}) on $\Omega$. Then, for $2 \leq k \leq n+1,$
    \begin{equation}\label{volumeconstraint}
        \sigma_{k}(\Omega) \leq \sigma_{k}(\Omega_0) = \sigma_2(\Omega_0),
    \end{equation} 
    where $\Omega_0= B_{R_2} \backslash \overline{B_{R_1}}$, the concentric annulus in $\mathbb{R}^n$ 
    with the constraint that 
    Vol$(B_{R_2})$ = Vol$(\tilde \Omega)$. 
\end{theorem}
\begin{proof}
    Since $\sigma_i(\Omega_0) = \sigma_2(\Omega_0), 2 \leq i \leq n+1$, it is enough to prove that $\sigma_{n+1}(\Omega) \leq \sigma_{n+1}(\Omega_0) = \sigma_2(\Omega_0)$. In order to prove this inequality, we plug in a certain test functions in the variational charactrization (\ref{charactrizaton}) of $\sigma_{n+1}(\Omega)$. Consider the following $(n+1)$-dimensional subspace of $\Tilde{H}_ 0^1(\Omega)$,     
    \begin{equation*}
        E= \text{span} \bigg\{ f_1
        , \frac{x_1
        }{\|x\|}f_1 \dots \frac{x_n
        }{\|x\|}f_1 \bigg\},
    \end{equation*}
where $f_1$ is as define in (\ref{eqn: eigenfuction}) with $l=1$. Now, For any $u \in E \backslash \{0\}$, there exist $c_0, c_1, \dots c_n \in \mathbb{R},$ not simultaneously equal to zero, such that 
    \begin{equation*}
        u=c_0\, f_1 + c_1\,\frac{x_1}{\|x\|}f_1 + \dots + c_n\, \frac{x_n}{\|x\|}f_1.
    \end{equation*}
Then, using Corollary \ref{cor:integral expression}, we get 
   \begin{align}
       \frac{\displaystyle \int_\Omega \| \nabla u \|^2 dV}{\displaystyle \int_{\partial \tilde{\Omega}} u^2 dS}= \frac{c_0^2 \displaystyle \int_\Omega \|\nabla f_1(\|x\|)\|^2 \, dV+ \displaystyle \sum_{i=1}^nc_i^2 \displaystyle \int_\Omega \bigg \| \nabla \left( \frac{f_1(\|x\|)}{\|x\|}  x_i \right) \bigg \|^2 \, dV}{c_0^2 \displaystyle \int_{\partial \tilde{\Omega}}f_1^2(\|x\|) \, dS + \displaystyle \sum_{i=1}^n c_i^2 \displaystyle \int_{\partial \tilde{\Omega}} \frac{f_1^2(\|x\|)}{\|x\|^2}x_i^2 \, dS}. \label{equality}
   \end{align} 
By Lemma \ref{constant lemma}, there exist constants $A_1, A_2 > 0$ such that for all natural numbers $1 \leq i \leq n$,
 \begin{align*}
       \int_{\partial \tilde{\Omega}}\left( \frac{f_1(\|x\|)}{\|x\|}x_i \right)^2 dS & =\int_{\partial \tilde{\Omega}} \frac{f_1^2(\|x\|)}{\|x\|^2}x_i^2 dS = A_1,\\
       \int_\Omega \bigg \| \nabla \left( \frac{f_1(\|x\|)}{\|x\|} x_i \right) \bigg \|^2 dV & = \int_\Omega \left( \frac{(f'_1(\|x\|))^2}{\|x\|^2} x_i^2 - \frac{f_1^2(\|x\|)}{\|x\|^4}x_i^2 + \frac{f_1^2(\|x\|)}{\|x\|^2} \right) dV = A_2.
   \end{align*}
Therefore,
      $$ n\,A_1 = \sum_{i=1}^n \int_{\partial \tilde{\Omega}} \left( \frac{f_1(\|x\|)}{\|x\|}x_i \right)^2 \, dS = \int_{\partial \tilde{\Omega}} f_1^2(\|x\|) \, dS, $$ and
       $$n\,A_2  = \sum_{i=1}^n  \int_\Omega \left( \frac{(f'_1(\|x\|))^2}{\|x\|^2} x_i^2 - \frac{f_1^2(\|x\|)}{\|x\|^4}x_i^2 + \frac{f_1^2(\|x\|)}{\|x\|^2} \right) dV 
        = \int_\Omega \left( (f'_1(\|x\|))^2 + \frac{(n-1)}{\|x\|^2} f_1^2(\|x\|)\right) \, dV.$$
Thus, for all natural numbers $1 \leq i \leq n$, we have 
   \begin{equation} \label{sum 1} 
       \displaystyle \int_{\partial \tilde{\Omega}} \left( \frac{f_1(\|x\|)}{\|x\|}x_i \right)^2 \, dS  = A_1 = \frac{1}{n} \displaystyle \int_{\partial \tilde{\Omega}} f_1^2(\|x\|) \, dS
        \end{equation}
         \begin{equation} \label{sum 1.5} 
        \displaystyle \int_\Omega \bigg \| \nabla \left( \frac{f_1(\|x\|) x_i}{\|x\|}\right) \bigg \|^2 \, dV  = A_2 = \frac{1}{n} \displaystyle \int_\Omega \left( (f'_1(\|x\|))^2 + \frac{(n-1)}{\|x\|^2} f_1^2(\|x\|)\right) \, dV.
         \end{equation} 
Now, from (\ref{equality}), \eqref{sum 1} and  \eqref{sum 1.5}, we get 
   \begin{align}
       \frac{\displaystyle \int_\Omega \| \nabla u \|^2 dV}{\displaystyle \int_{\partial \tilde{\Omega}} u^2 \, dS} = \frac{c_0^2 \displaystyle \int_\Omega \|\nabla f_1(\|x\|)\|^2 \, dV + A_2 \displaystyle \sum_{i=1}^n c_i^2}{c_0^2 \displaystyle \int_{\partial \tilde{\Omega}}f_1^2(\|x\|) \, dS + A_1 \displaystyle \sum_{i=1}^n c_i^2 }
        \leq \text{max} \Bigg \{\frac{\displaystyle \int_\Omega \|\nabla f_1(\|x\|)\|^2 \, dV}{\displaystyle \int_{\partial \tilde{\Omega}}f_1^2(\|x\|) \, dS}, \frac{A_2}{A_1} \Bigg \}. \label{gradient inequality 1}
   \end{align}
Now 
\begin{align*}
    \frac{A_2}{A_1} & = \frac{\displaystyle \int_\Omega \left( (f'_1(\|x\|))^2 + \frac{(n-1)}{\|x\|^2} f_1^2(\|x\|)\right) dV}{\displaystyle \int_{\partial \tilde{\Omega}} f_1^2(\|x\|) \, dS} 
                     \geq \frac{\displaystyle \int_\Omega (f'_1(\|x\|))^2 \, dV}{\displaystyle \int_{\partial \tilde{\Omega}} f_1^2(\|x\|) \, dS} 
                     = \frac{\displaystyle \int_\Omega \|\nabla f_1(\|x\|)\|^2 \, dV}{\displaystyle \int_{\partial \tilde{\Omega}}f_1^2(\|x\|) \, dS}.
\end{align*}
   Then from the inequality (\ref{gradient inequality 1}) we get
   \begin{align}
       \frac{\displaystyle \int_\Omega \| \nabla u \|^2 dV}{\displaystyle \int_{\partial \tilde{\Omega}} u^2 \, dS} \leq  \frac{A_2}{A_1} =\frac{\displaystyle \int_\Omega \left( (f'_1(\|x\|))^2 + \frac{(n-1)}{\|x\|^2} f_1^2(\|x\|)\right) \, dV }{\displaystyle \int_{\partial \tilde{\Omega}} f_1^2(\|x\|) \, dS}. \label{gradient inequality}
    \end{align}
Next, using the Lemmas \ref{lem: integral on domain} and \ref{lem:integral on boundary}, we get  
   \begin{align*}
       \frac{A_2}{A_1} =\frac{ \displaystyle \int_\Omega \left( (f'_1(\|x\|))^2 + \frac{(n-1)}{\|x\|^2} f_1^2(\|x\|)\right) \, dV }{\displaystyle \int_{\partial \tilde{\Omega}} f_1^2(\|x\|) \, dS} \leq \frac{\displaystyle \int_{\Omega_{0}} \left( (f'_1(\|x\|))^2 + \frac{(n-1)}{\|x\|^2} f_1^2(\|x\|)\right) \, dV }{\displaystyle \int_{\partial B_{R_2} }f_1^2(\|x\|) \, dS} = \sigma_2(\Omega_0).
   \end{align*}
Therefore, from the variational charactrization (\ref{charactrizaton}) and inequality (\ref{gradient inequality}), we conclude,
   \begin{align}
      \sigma_{(n+1)}(\Omega) \leq \max_{u(\neq 0) \in E} \frac{\displaystyle \int_\Omega \| \nabla u \|^2 \, dV}{\displaystyle \int_{\partial \tilde{\Omega}} u^2 \, dS} \leq \sigma_2(\Omega_0).
   \end{align}
This completes the proof of Theorem \ref{Thm:higher eigenvalues}.
\end{proof}

\begin{remark}
Take $\tilde{\Omega}$ to be the open ellipse with major axis of length $9.5$ cm and minor axis of length $10.526$ cm centered at the origin. Let the ball $B:=B_2((3,5))$ be the ball centered at $(3,5)$ having radius $2$ cm be our inner domain $B_{R_1}$. Then, the second and the third Steklov-Dirichlet eigenvalues of the domain $\Omega = \tilde{\Omega} \setminus \overline{B}$ have the following values: $\sigma_2(\Omega) = 0.110414$ and $\sigma_3(\Omega) = 0.15868$. However, if we consider the concentric annual domain with outer ball $B_{10}$, a ball centered at the origin having radius $10$ cm 
and, with inner ball $B_2$, again centered at the origin, having radius $2$ cm, then $\sigma_2 (B_{10} \setminus \overline{B_2} ) = \sigma_3 (B_{10} \setminus \overline{B_2} ) = 0.108334$. Note here that, Vol$(\Omega)$ = Vol$(B_{10} \setminus \overline{B_2})$ and $\sigma_2(\Omega) > \sigma_2 (B_{10} \setminus \overline{B_2} )$. Thus we observe that the annular domain does not maximize the 2nd and 3rd eigenvalue among all doubly connected domains of fixed volume even if radius of inner ball is fixed. 
Therefore Theorem \ref{Thm:higher eigenvalues} may not hold if we drop the symmetry assumption on the domain $\Omega$.
\end{remark}

\section{Concluding Remarks} \label{sec:remark on star shaped}
In this section, we state a 
generalization of Theorem  \ref{Thm:upper bound in convex domains}. 
\begin{enumerate}
    \item The following theorem extends Theorem \ref{Thm:upper bound in convex domains} of \cite{gavitone2023isoperimetric} from bounded convex domains to bounded star shaped domains in $\mathbb{R}^{n}, n \geq 2$. It can be proved using exactly the same arguments.

\begin{theorem} \label{Thm:upper bound}
 Let $\tilde{\Omega}$ 
 be a bounded domain in $\mathbb{R}^{n}$ which is star shaped with respect to the origin. Let $B_{R_1} \subset \mathbb{R}^{n}$ be a ball of radius $R_1$ centered at the origin such that $\overline{B_{R_1}} \subset \tilde{\Omega}$. Define \begin{equation*}
     \overline{R}_1 := 
     \begin{cases}
     \begin{aligned}
& R_1 e^{\sqrt{2}} & \text{ if }  n=2,\\
& R_1 \left[ \frac{(n-1) + (n-2) \sqrt{2(n-1)}}{n-1} \right]^{\frac{1}{n-2}} & \text{ if }  n \geq 3.
 \end{aligned}
     \end{cases}
 \end{equation*} Let $B_{\overline{R}_1}$ 
 be the ball in $\mathbb{R}^{n}$ centered at the origin with radius $\overline{R}_1$. We further assume that $\tilde{\Omega} \subset B_{\overline{R}_1}$, 
Then,
 \begin{equation*}
     \begin{aligned}
\sigma_1 (\tilde{\Omega} \backslash \overline{B_{R_1}}) \leq \sigma_1 (B_{R_2} \backslash \overline{B_{R_1}}),  
     \end{aligned}
 \end{equation*}
 where $B_{R_2}$ is the ball of radius $R_2$ such that $\overline{B_{R_1}} \subset B_{R_2}$ and Vol$(\tilde{\Omega})$ = Vol$(B_{R_2})$.
\end{theorem}

\item In future, we are planning to prove Theorems \ref{Thm:higher eigenvalues} and \ref{Thm:upper bound} for domains contained in simply connected space forms.
\end{enumerate}

\textbf{Acknowledgement:} S Basak is supported by University Grants Commission, India. A Chorwadwala acknowledges the SERB MATRICS grant sanction order No. MTR/2019/001309. The corresponding author S. Verma acknowledges the project grant provided by SERB-SRG sanction order No. SRG/2022/002196. 

\bibliographystyle{plain}
\bibliography{Ref}

\end{document}